\documentclass[10pt]{article}
\usepackage[margin=1in]{geometry}
\usepackage[T1]{fontenc}
\usepackage[utf8]{inputenc}
\usepackage[english]{babel}
\usepackage{microtype}
\usepackage[dvipsnames]{xcolor}
\usepackage{float}
\usepackage{comment}
\usepackage{graphicx}
\usepackage{titling}
\usepackage{caption}
\usepackage{subcaption}
\usepackage{soul}
\usepackage{enumitem}
\usepackage{manfnt}
\usepackage{ upgreek }
\usepackage[hang,flushmargin]{footmisc}
\usepackage{xcolor}
\usepackage{tikz}

\usepackage{amsmath, amsthm, amssymb, amsfonts, mathtools}
\usepackage{bbm}

\theoremstyle{plain}
\newtheorem{lemma}{Lemma}[section]
\newtheorem{theorem}[lemma]{Theorem}

\newtheorem{claim}[lemma]{Claim}

\theoremstyle{remark}
\newtheorem{remark}[lemma]{Remark}

\newtheorem{question}[lemma]{Question}

\newcommand{\Z}{\mathbb{Z}}

\newcommand{\N}{\mathbb{N}}

\newcommand{\mcT}{\mathcal{T}}

\usepackage{hyperref}
\numberwithin{equation}{section}

\title{Graphings with few circulations}
\author{
Gábor Kun 
\and 
László Márton Tóth 
}
\date{}

\begin{document}

\maketitle
\setcounter{tocdepth}{2}
\begin{abstract} 
In 2021, motivated by graph limit theory Lovász extended most of the theory of flows to a measure theoretic setting \cite{lovasz2021flows}. Using this framework, the first author constructed $d$-regular treeings that are measurably bipartite, and have no nonzero measurable circulations, that is, flows without sources or sinks \cite{kun2021measurable}. In particular, these treeings do not admit a measurable perfect matching.

In this paper, we develop tools to build $d$-regular treeings where the space of circulations is exactly $ k$-dimensional for any positive integer $k$. As applications, we construct
\begin{enumerate}
    \item a treeing with a single balanced orientation, but no Schreier decoration;
    \item a treeing with a single Schreier decoration;
    \item and a treeing with a proper edge $d$-coloring, but no further perfect matchings.
\end{enumerate}
The first answers a question raised by Lovász, as this particular balanced orientation does not decompose as a linear combination of finite cycles and infinite paths. 
\end{abstract}

\tableofcontents

\newpage
\section{Introduction} \label{sec:intro}

Measurable combinatorics is primarily concerned with solving graph theoretic problems on Borel graphs, see, e.g., Marks' survey \cite{marks2022icm}. Probability measure preserving (pmp) graphs called \emph{graphings} play an important role in measured group theory and graph limit theory, see Gaboriau \cite{gaboriau2010icm} and Lov\'asz \cite{lovasz2012large}. 

Perfect matchings have attracted special attention due to their applications to equidecompositions, in particular, to Tarski's circle squaring problem, see M\'ath\'e's survey \cite{mathe2018icm}.
The main question in the bipartite measurable context is deciding when the measurable Hall condition implies the existence of a measurable perfect matching.

Laczkovich \cite{laczkovich1988closed} showed an example of a measurably bipartite $2$-regular graphing admitting no measurable perfect matchings. Bowen, the first author and Sabok proved a measurable Hall theorem for hyperfinite graphings, \cite{bowen2021perfect} showing that the only possible obstructions to the existence of a measurable perfect matching are two-ended subgraphing (i.e., of linear growth) without a measurable perfect matching.
In the non-hyperfinite case there are also notable positive results on the existence of measurable perfect matchings, see Lyons and Nazarov \cite{lyons2011perfect}, Cs\'oka and Lippner \cite{csoka2017invariant}.
The introduction of the first author \cite{kun2021measurable} gives a more detailed overview of recent results on measurable perfect matchings. 

In measurably bipartite graphings perfect matchings are closely related to \emph{circulations}, i.e., flows without sinks or sources. Every perfect macthing gives rise to a circulation defined by setting its value to be $d-1$ on the matching edges and $-1$ outside the matching. This connection has allowed the first author \cite{kun2021measurable} to 
construct for every $d \geq 3$ a $d$-regular \emph{treeing} (i.e., essentially acyclic graphing) that admits no measurable perfect matching,
refuting a question of Kechris and Marks 
\cite{kechris2016descriptive}. The construction uses a diagonalization argument to get a treeing without any nonzero circulation. In this paper, we adopt this strategy and highlight its flexibility in producing examples where a few circulations are permitted to survive.

\begin{theorem} \label{thm:finitely_many_circ}
   For every degree $d \geq 3$ and finite number $k \in \mathbb{N}$ there exists a $d$-regular treeing $\mcT$ such that the vector space of circulations in $L^2\big(E(\mcT), \eta \big)$ is exactly $k$ dimensional. 
\end{theorem}

Here $\eta$ denotes the edge measure of $\mcT$, see Section~\ref{sec:prelim}. Using the same argument, we can also build treeings that carry certain combinatorial structures, such as a balanced orientation, a Schreier decoration, or a proper $d$-edge-coloring, but admit only one of them. The proofs of the three theorems treating these cases are slight adjustments of the proof of Theorem~\ref{thm:finitely_many_circ}.

An orientation of a $2d$-regular graph is called \emph{balanced} if all in- and outdegrees are equal to $d$. In a \emph{Schreier decoration}, in addition to a balanced orientation, we also have a labeling of the oriented edges with $d$ labels such that for each label all in- and outdegrees are 1. We can view a balanced orientation as a circulation with $\pm1$ values, and a Shreier decoration as a specific type of decomposition of a balanced orientation into the sum of $d$ circulations. Schreier decorations encode actions of the free group $\mathbb{F}_d$. Theorems for balanced orientations and perfect matchings go usually hand in hand. For results on finding measurable balanced orientations and Schreier decorations on graphings see the second author \cite{toth2021invariant}, Bencs, Hru\v{s}kov\'a, and the second author\cite{bencs2024factor_euclidean, bencs2024factor_nonamen}, Thornton \cite{thornton2022orienting}, and Bowen, the first author, and Sabok  \cite{bowen2021perfect}.

First, we allow a balanced orientation, but nothing else.
\begin{theorem} \label{thm:balanced_orientation}
For every $d\geq2$ there exists a $2d$-regular treeing that has a measurable balanced orientation, but admits no nonzero measurable circulations in $L^2\big(E(\mcT), \eta\big)$ besides constant multiples of the balanced orientation. In particular, it has no measurable Schreier decoration, i.e., it is not generated by an essentially free action of the free group $\mathbb{F}_d$. 
\end{theorem}
This answers a question raised in a paper of Lovász [Remark 4.19,\cite{lovasz2021flows}], namely, if circulations can always be decomposed into a nonnegative linear combination of finite directed cycles and directed two-way infinite paths. The balanced orientation of the treeing in Theorem~\ref{thm:balanced_orientation} admits no such decomposition. 

Next, we allow one single Schreier decoration, but nothing else.

\begin{theorem} \label{thm:Schreier_graphing}
For every $d\geq2$ there exists a $2d$-regular treeing that is a Schreier graphing of $\mathbb{F}_d$, but admits no nonzero measurable circulations in $L^2\big(E(\mcT), \eta \big)$ besides linear combinations of the ones corresponding to the generators of $\mathbb{F}_d$.
\end{theorem}

Similarly to Schreier decorations, a proper $d$-edge-coloring encodes an action of the group $(\mathbb{Z}/2\mathbb{Z})^{*d}$, the free product of $d$ two-element groups. It is the union of $d$ pairwise disjoint perfect matchings. The next theorem allows such a proper $d$-edge-coloring, but no further perfect matchings.

\begin{theorem} \label{thm:perfect_matching}
For every $d\geq3$ there exists an essentially free pmp action of the free product $(\Z / 2\Z)^{*d}$ such that its Schreier graphing has exactly $d$ perfect matchings, the ones corresponding to the generators.
\end{theorem}

\paragraph{Organization of the paper.} In Section~\ref{sec:prelim} we introduce the core definitions and results we use. In Section~\ref{sec:finitely_many_circ} we prove our most general result, Theorem~\ref{thm:finitely_many_circ}. In Section~\ref{sec:applications} we show how one needs to modify the construction to get Theorems~\ref{thm:balanced_orientation}, \ref{thm:Schreier_graphing}, and \ref{thm:perfect_matching}. In Section~\ref{sec:further_directions} we discuss a few possible directions for future research.

\paragraph{Acknowledgements.} The authors are grateful to Miklós Abért,  Lukasz Grabowski, László Lovász, and Zoltán Vidnyánszky for helpful discussions.

Gábor Kun has been supported by the Hungarian Academy of Sciences Momentum Grant no. 2022-58 and ERC Advanced Grant ERMiD.
László Márton Tóth has been supported by the National Research, Development and Innovation Fund -- grants number KKP-139502 and STARTING 150723. 

\section{Preliminaries} \label{sec:prelim}

\subsection{Graphings and circulations}
Recall that a \emph{graphing} $\mathcal{G}$ is a graph whose vertex set is a standard Borel probability measure space $(X, \lambda)$, and whose edge set is the countable union of graphs of pmp partial functions on $X$. A \emph{treeing} is a graphing that is essentially acyclic. In this paper we only work with $d$-regular graphings for a finite $d>0$. We consider the \emph{edge measure} of $\mathcal{G}$ induced by $\lambda$ as follows: 

$$\eta(S) = \frac{1}{d}\int_X \deg_S(x) \ d\lambda(x)$$ 
for each $S \subseteq E(\mathcal{G})$ Borel, where $\deg_S(x)$ denotes the degree of $x$ using edges from $S$.  Note that 
$\int_X \deg(x) \ d\lambda(x)=d$, hence our choice of normalization makes $\eta$ a probability measure. This is how the edge measure is normalized in Markov spaces in the works motivating our paper \cite{lovasz2021flows, kun2021measurable}. It is a slight deviation from the convention for graphings, where the edge measure is usually normalized differently. See, e.g., \cite{lovasz2012large} for a detailed introduction to graphings. 

Since we work in the measurable context we might understand properties of graphings upto a nullset: the degrees are countable, hence the union of the connected components intersecting a nullset is still a nullset.

For a graph $G$ we denote its edge set by $E(G)$. We only work with undirected graphs in this paper, but by $E(G)$ we mean a set of directed edges, i.e., for every pair of adjacent vertices $u, v$ we have $(u,v),(v,u) \in E(G)$.
A \emph{circulation} on a graph $G$ is an antisymmetric function $f: E(G) \to \mathbb{R}$ that satisfies the flow condition at each vertex, i.e., for every $x \in V(G)$ \[\sum_{y: (x,y) \in E(G)} f(x,y) = \sum_{y: (y,x) \in E(G)} f(y,x)=0.\]
Note that the equality of the two sums follows from the antisymmetry.

A \emph{potential} is a function $h: V(G) \to \mathbb{R}$. We denote its gradient by $\nabla h : E(G)\to \mathbb{R}$, i.e., $\nabla h(x,y) = h(y)-h(x)$. Circulations are characterized by the fact that they are orthogonal to all gradients of potentials. This is due to Lovász in the context of measure valued flows on measure spaces \cite{lovasz2021flows}. Here we state it for graphings.

\begin{lemma}\label{lem:potential_vs_circulation}
    The function $f:E(\mathcal{G}) \to \mathbb{R}$ is a circulation if and only if
    \[\int_{E(\mathcal{G})} f(x,y)\nabla h(x,y) \ d \eta(x,y) =0\]
     holds for all bounded potentials $h$.
\end{lemma}

The circulations we permit to survive in the limit will be rational valued. It will be convenient to adopt the following notations:

\begin{itemize}
    \item $[K]= \{1, \ldots , K\};$
    \item $[K]_i=\left\{\frac{p}{q} ~\Big|~ p,q \in \Z, q \leq i, 
    0 < p/q \leq K\right\}$.
\end{itemize}

\subsection{Inverse limits}

Let $\{ G_n \}_{n=1}^{\infty}$ be a sequence of finite graphs and $f_n:V(G_{n+1}) \rightarrow V(G_n)$ be a sequence of graph homomorphisms. We define the inverse limit of the sequence $\{ G_n ,f_n\}_{n=1}^{\infty}$ to be the graph $G$ with vertex set 
	\[ V(\mathcal{G})=\big\{ (x_n)_{n=1}^{\infty}:\forall n \text{ } x_n \in V(G_n), f_n (x_{n+1})=x_n  \big\}\]
	and edge set
	\[E(\mathcal{G})=\{ (x, y): x,y \in V(G), \forall n  \ (x_n, y_n) \in E(G_n) \}.\] 
	We endow $V(\mathcal{G})$ with the topology inherited from the product topology of the discrete topologies on $V(G_n)$.
	
	By slightly abusing the notation if $f:V(G') \to V(G)$ is a mapping and $(u,v) \in E(G)$, we will denote by $f^{-1}(u,v)$ the set \[\{(u',v') \in E(G'): f(u')=u \text{ and } f(v')=v\}.\]
	
	A sequence $\{ G_n,f_n \}_{n=1}^{\infty}$ will be called $d$-\emph{proper}, if the graphs are $d$-regular and the preimages of both edges and vertices have the same cardinality. That is, for every $n$
	\begin{enumerate}
		\item the graph $G_n$ is $d$-regular,
        \item $|f_n^{-1}(u)|=\frac{|V(G_{n+1}|}{|V(G_n)|}$ for every $u \in V(G_n)$, and
		\item $|f^{-1}_n(u,v) \cap E(G_{n+1})|=\frac{|E(G_{n+1})|}{|E(G_n)|}$ for every $(u,v) \in E(G_n)$.
	\end{enumerate}
	
    Let $\mathcal{A}$ be the Borel $\sigma$-algebra on $V(\mathcal{G})$, then $\mathcal{A}$ is generated by cylinders of the form $\{ x: x_n \in S \}$, where $S \subseteq V(G_n)$ for some $n$. Set
	\[\lambda\big(\{ x: 
	x_n \in S \}\big):=\frac{|S|}{|V(G_n)|},\]
    this extends uniquely to $\mathcal{A}$.
 
    Similarly, consider the product $\sigma$-algebra $\mathcal{A}^2$ on $V(\mathcal{G})^2$. Define $\eta$ by \\ \[\eta\Big(\{ (x, y): (x_n, y_n) \in Q \}\Big):=\frac{|Q \cap E(G_n)|}{|E(G_n)|},\] for every $n$ and $Q \subseteq V(G_n)^2$, this uniquely extends to 
	$\mathcal{A}^2$. We collect the basic properties of the inverse limit of a $d$-proper sequence.

\begin{lemma} \label{basic} (Lemma 3.1 \cite{kun2021measurable})
	Let $\{ G_n,f_n \}_{n=1}^{\infty}$ be a $d$-proper sequence, and define $\mathcal{G}, \lambda$ and $\eta$ as above.  
	Then $\mathcal{G}$ is a $d$-regular graphing with the probability measure $\lambda$ on the set of vertices, and the induced probability measure $\eta$ on the set of edges.  Moreover, if $G_1$ is bipartite then $\mathcal{G}$ admits a clopen bipartition.
\end{lemma}

\section{Finite number of circulations} \label{sec:finitely_many_circ}

In this section, we prove our most general result, Theorem~\ref{thm:finitely_many_circ}. As the argument gets fairly technical, we first outline our strategy in~\ref{subsec:strategy} while in~\ref{subsec:glossary} we collect the parameters we will need to keep track of. We detail our construction  in~\ref{subsec:construction}, and in~\ref{subsec:no_other_circ} we finally prove that the treeing we build has no unwanted circulations. 

\subsection{Strategy} \label{subsec:strategy}

We construct $\mcT$ as an inverse limit of finite graphs that almost cover each other. We start with a finite $d$-regular graph $G_1$ and $\alpha^1_1, \ldots, \alpha^1_k$ pairwise orthogonal, rational valued circulations on $G_0$. At each step, we
aim to build the next graph $G_{n+1}$ in a way that the $\alpha^n_i$ lift to it, but the circulations $\beta \perp \langle \alpha^n_i:1\leq i \leq k \rangle $ do not. We achieve this by 
\begin{enumerate}
    \item identifying a finite, but sufficiently dense subset $\mathcal{B}_n$ of such $\beta$'s that we aim to eliminate;
    \item introducing separate coordinates for all $\beta \in \mathcal{B}_n$;
    \item building a topological cover of $G_n$ by adding coordinate vectors $\mathbf{x}=\big(x(\beta)\big)_{\beta \in \mathbf{B}_n}$ to vertices of $G_n$, where we connect  $(u,\mathbf{x})$ to $(v,\mathbf{y})$ if $(\mathbf{y}-\mathbf{x})(\beta) = \beta(u,v)$ for all $\beta \in \mathcal{B}_n$; (This  cover corresponds to the homomorphism of the fundamental group of $G_n$ into $\mathbb{Q}^{\mathcal{B}_n}$ induced by the set $\mathcal{B}_n$.)
    \item cutting sufficiently large boxes with small relative boundary out of this infinite cover by restricting the values of the coordinates to be in an interval;
    \item \label{step:surgery} adding a relatively small number of vertices and edges near the boundary of the boxes to maintain $d$-regularity. (Surgery step.)
\end{enumerate} 

We denote by $\alpha_i$ the limit of the sequence $\{ \alpha^n_i \}_{n=1}^{\infty}$. This exists, since we take care that the sequence stabilizes almost everywhere. We suppose for a contradiction that there is an $L^2$ circulation $\gamma$ on $\mcT$ with norm one that is orthogonal to all the $\alpha_i$. By the properties of our construction, we can find a $\beta$ at a finite level $G_n$ approximating $\gamma$ very well. Crucially, at the next level, projecting $G_{n+1}$ to the coordinate corresponding to $\beta$ defines a potential $h_{\beta}$ whose gradient coincides with $\beta$ on all but a small proportion of the edges, those added during the surgery. We can lift $h_{\beta}$ to a potential $h$ on $\mcT$ whose gradient is very close to $\gamma$. This leads to a contradiction, as, on one hand, $\langle \gamma , \nabla  h \rangle$ should be close to $\|\gamma\|_2=1$, while, on the other hand, by Lemma \ref{lem:potential_vs_circulation} it is close to zero, since it is the inner product of a circulation and the gradient of a potential.

A few remarks about this strategy are in order.

\begin{remark}
    At every stage, we get to choose the size of the box we want to cut out. We choose this to be sufficiently large, even compared to the parameters of the previous graph. Thus, all parameters designed to bound approximation errors will be arbitrarily small.
\end{remark}

\begin{remark}
It is important that we do not build topological covers at the finite stages, since that would lead to all circulations lifting. Even though the proportion of vertices added during the surgery is small, this part is responsible for separating the circulations that we keep from those we eliminate. 
\end{remark}

\begin{remark} 
We will make an effort to ensure that $G_1, G_2, \ldots$ is a proper sequence. This makes it necessary during Step \ref{step:surgery} of the strategy above to extend the homomorphism and also to construct the preimages of the edges such that they have the same size. This involves technical assumptions on $\mathcal{B}_n$, arguments to make $G_n$ connected and non-bipartite, and using gadgets to artificially increase preimages of certain edges. 

These considerations are not essential for the construction to work. Alternatively, we could work with an almost proper sequence and its inverse limit. In that case, however, keeping track of all approximations during the proof of Theorem \ref{thm:finitely_many_circ} would become quite cumbersome. We chose to address the combinatorial complications during the construction instead, and consequently make the computations later simpler.
\end{remark}

\begin{remark}
The novel ingredients of our strategy (compared to those introduced in \cite{kun2021measurable}) are 
\begin{itemize}
    \item working with $l^2$ edge spaces instead of using partial orientations to eliminate circulations;
    \item selecting the circulations to be eliminated by picking $\mathcal{B}_n$ from the orthogonal complement of the $\alpha_i^n$;
    \item establishing a Deficiency Lemma (Lemma~\ref{lemma:deficiency}) which will allow us to perform the surgery on the boundary mostly locally.
\end{itemize}
\end{remark}

\subsection{Glossary} \label{subsec:glossary}
Before we describe the construction in detail, we collect the parameters that will appear.

\begin{itemize}
    \item $\varepsilon >0$ will be the ultimate error of the approximation of the inner product $\langle \gamma, \nabla h\rangle$.
    \item $D$ will denote the common denominator of the entries of the  vectors $\alpha^n_i$.
    \item $M$ will bound the maximal absolute value of entries of the vectors $\alpha^n_i$.
    \item $M_n$ will bound the maximal value of the entries of vectors that we intend to eliminate at step $n$.
    \item $D_n$ will denote the common denominator needed at step $n$. 
    \item $\theta_n$ will bound the proportion of vertices added during the surgery at step $n$.
    \item $K_n$ will determine our choice of how large boxes we use when building the finite almost-covers. The side of the box will have length $K_n M_n$. We will choose $K_n$ large enough such that all approximations hold with little error.
\end{itemize}

\subsection{Construction} \label{subsec:construction}

\subsubsection{General setup}
Fix $\varepsilon > 0$ small enough to be specified later. We construct $\mcT$ as an inverse limit of finite graphs that almost cover each other. Let $G_1$ be a suitable finite connected non-bipartite $d$-regular graph such that there are $\alpha^1_1, \ldots, \alpha^1_k$ pairwise orthogonal, rational valued circulations on $G_1$ with
norm equal to one. Let $D$ denote the common denominator of their entries, $M = \max_i \|\alpha^1_i\|_{\infty}$, and $K_0=M_0=1$.

The circulations $\alpha_i^{n+1}$ that we construct will be almost orthogonal:

$$\forall \text{ } i \neq j \text{ } |\langle \alpha_i^{n+1}, \alpha_j^{n+1} \rangle | \leq \frac{1}{2k}-\frac{1}{2k+2n}.$$ 

This holds for $n=0$. Given $G_n$ and $\alpha^n_1, \ldots, \alpha^n_k$, we define $G_{n+1}$ as follows. The vertex set $V(G_{n+1})$ will consist of two disjoint sets, denoted by $V_0(G_{n+1})$ and $U(G_{n+1})$ such that $\frac{|U(G_{n+1})|}{|V(G_{n+1})|}$ will be less then a carefully chosen parameter $\theta_n$. We also choose an even $M_n \in \mathbb{N}$ that will bound the entries of vectors treated at this stage. 
Let $W_n=\langle \alpha^n_1, \ldots, \alpha^n_k\rangle^{\perp}$ in $\ell^2(E(G_n))$ and $\mathcal{B}_n$ be a finite $\frac{\varepsilon}{2}$-net of vectors with rational coefficients in $W_n \cap [-M_n, M_n]^{E(G_n)}$. 
To make the proof of connectivity of $G_{n+1}$ easier, we also need a technical condition on $\mathcal{B}_n$. We need that for every $x \in V(G_n)$ there exist two different $\beta \in \mathcal{B}_n$ such that $\beta$ takes both a positive and a negative value on some edge starting at $x$. We can choose such $\beta$'s even in the smaller subspace orthogonal to every circulation on $G_n$, that is, as gradients of potentials.

Denote the common denominator of the entries of the elements of $\mathcal{B}_n$ by $D_n$. We construct $V_0(G_{n+1})$ by adding coordinates to the vertices of $G_n$.  We introduce a coordinate for each $\beta \in \mathcal{B}_n$, and allow the  coordinates to take rational values from a large interval. We choose the length of the interval to be a multiple of $M_n$. Set

\[V_0(G_{n+1}) = V(G_n) \times \left([K_n \cdot M_n]_{d^2D_n}\right)^{\mathcal{B}_n},\] 
where the multiplier $K_n \in \N$ is chosen large enough, even compared to $|\mathcal{B}_n|$.

Let $p^n_0$$: V_0(G_{n+1}) \to V(G_n)$ and $p^n_{\beta}: V_0(G_{n+1}) \to [K_n \cdot M_n]_{d^2D_n}$ denote the projections to the coordinates. As pointed out in~\ref{subsec:strategy}, the role of these coordinates corresponding to $\mathcal{B}_n$ is to eliminate all circulations orthogonal to the $\alpha_i^n$ already on the next graph. We first introduce the set of edges $E_0(G_{n+1})$ on $V_0(G_{n+1})$. (The set of vertices $U(G_{n+1})$ and the edges incident to it will be constructed separately in Lemma~\ref{lemma:extension}.) Let $u,v \in V(G_n)$ and $\mathbf{x}, \mathbf{y} \in \left([K_n\cdot M_n]_{d^2D_n}\right)^{\mathcal{B}_n}$. Then $(u,\mathbf{x})$ and $(v, \mathbf{y})$ form an edge in $E_0(G_{n+1})$ if and only if $u$ and $v$ are neighbors in $G_n$ and $\mathbf{x}$ and $\mathbf{y}$ differ by the appropriate amount, namely $\beta(u,v)$, in every coordinate. That is, we require 

\begin{equation} \label{eqn:edge_lift}
(u,v) \in E(G_n) \textrm{, and } \ \mathbf{y}(\beta)- \mathbf{x}(\beta) = \beta(u,v) \quad \forall \beta \in \mathcal{B}_n.
\end{equation}

Note that $\big(V_0(G_{n+1}),E_0(G_{n+1})\big)$ is not connected, since for any two vertices in the same component the difference of the coordinates corresponding to the same projection $p^n_{\beta}$ is an integer multiple of $\frac{1}{D_n}$. Consider the induced subgraph on the set of vertices

$$\underline{V}_0(G_{n+1})=V(G_n) \times \left([K_n \cdot M_n]_{D_n}\right)^{\mathcal{B}_n}.$$

Denote its edge set by $\underline{E}_0(G_{n+1})$. The graph $\big(V_0(G_{n+1}), E_0(G_{n+1})\big)$ is the disjoint union of $d^{2|\mathcal{B}_n|}$ isomorphic copies of $\big(\underline{V}_0(G_{n+1}), \underline{E}_0(G_{n+1})\big)$, and the gradient of every $p^n_{\beta}$ is invariant under the isomorphisms between these copies. The purpose of the $d^2$ term in $d^2D_n$ is to yield many isomorphic copies. We will make use of this when we perform the same surgeries on these copies in order to construct $G_{n+1}$.

\subsubsection{Deficiency near the boundary}
\label{subsubsec:ingredients}

Our eventual aim is to construct $U(G_{n+1})$ and $\alpha^{n+1}_i$. We start by lifting the $\alpha^n_i$ to $E_0(G_{n+1})$ by setting $\alpha^{n+1}_i(a,b) = \alpha^n_i\big(p_0^n(a), p_0^n(b)\big)$, and analyzing where, and by what amount they fail to be circulations.
Partition $V_0(G_{n+1})$ into $M_n$-by-$M_n$ cubes according to the second coordinate. That is, for an integer-valued vector $\mathbf{m} \in [K_n]^{\mathcal{B}_n}$ let \[V_{\mathbf{m}} = \left\{(v, \mathbf{x}) \in V_0(G_{n+1}) ~\Bigg|~ M_n\cdot \big(\mathbf{m(\beta)}-1 \big) <  \mathbf{x}(\beta) \leq M_n\cdot \mathbf{m}(\beta), \ \forall \beta \in \mathcal{B}_n\right\}.\] 

First we consider $M_n$-by-$M_n$ cubes that are not on the boundary, i.e., we assume $\mathbf{m} \in \{2,\ldots,K_n-1\}^{\mathcal{B}_n}$. Let $(u, \mathbf{x}) \in V_{\mathbf{m}}$. Every $G_{n}$-edge $(u,v)$ has, at the vertex $(u, \mathbf{x})$, a unique lift $((u, \mathbf{x}),(v,\mathbf{y}))$. Indeed, all the coordinates of $\mathbf{y}$ are determined by equation (\ref{eqn:edge_lift}). As $M_n < \mathbf{x}(\beta) \leq M_n(K_n-1)$, and  $\beta(u,v) \in [-M_n,M_n]$ for all $\beta \in \mathcal{B}_n$, we have $0 < y(\beta) \leq K_nM_n$, so $(v,\mathbf{y})$ is indeed a vertex of $V_0(G_{n+1})$. In other words, the map $p_0^n$ is a local isomorphism at the vertices of $V_{\mathbf{m}}$. Consequently, such vertices have degree $d$, and all the $\alpha^{n+1}_i$ satisfy the flow condition at these vertices. 

This fails, however, near the boundary, i.e., when $\mathbf{m}$ has a coordinate equal to $1$ or $K_{n}$. We call such coordinates of $\mathbf{m}$ \emph{tight}. As a first step towards constructing $U(G_n)$ we show that for those $V_{\mathbf{m}}$ where $\mathbf{m}$ has only one tight coordinate, the total amount missing at vertices cancels out, so there is no obstruction to doing the surgery locally. More precisely, introduce the $\alpha^{n+1}_i$-\emph{deficiency} of a vertex as
\[\texttt{def}_{\alpha^{n+1}_i}(u,\mathbf{x}) = \sum_{\big((u, \mathbf{x}),(v,\mathbf{y})\big) \in E_0(G_{n+1})} \alpha^{n+1}_i \big((u,\mathbf{x}),(v, \mathbf{y})\big).\] The deficiency of a vertex expresses how much weight $\alpha^{n+1}_i$ we will need to assign in total to the extra edges that we will introduce at $(u,\mathbf{x})$ in order to make $\alpha^{n+1}_i$ a circulation. The deficiency of a vertex set $A$ is defined as the sum of the deficiencies of the vertices: \[\texttt{def}_{\alpha^{n+1}_i}(A) = \sum_{(u,\mathbf{x}) \in A} \texttt{def}_{\alpha^{n+1}_i}(u,\mathbf{x}).\]

\begin{lemma}[Deficiency lemma] \label{lemma:deficiency}
    If $\textbf{m}$ has one single tight coordinate, then $\texttt{def}_{\alpha^{n+1}_i}(V_{\mathbf{m}}) =0$ for all $i \in \{1,\ldots, k\}$. 
\end{lemma}

\begin{proof}
    Let $\beta$ stand for the tight coordinate of $\mathbf{m}$. Without loss of generality we can assume $\mathbf{m}(\beta)=1$, the proof is similar in the $\mathbf{m}(\beta)=K_{n}$ case. Since $\beta \in \mathcal{B}_n$, we have
    \[0 = \langle \alpha^n_i, \beta \rangle = \frac{1}{|E(G_n)|} \sum_{(u,v)\in E(G_n)} \alpha^n_i(u,v) \beta(u,v).\]

    As we argued before, the edge $(u,v) \in E(G_n)$ has no lift from $(u,\mathbf{x}) \in V_{\mathbf{m}}$ if and only if $\mathbf{x}(\beta) \leq \beta(u,v)$. Therefore,
    \[\texttt{def}_{\alpha^{n+1}_i}(u,\mathbf{x}) = \sum_{\big((u, \mathbf{x}),(v,\mathbf{y})\big) \in E_0(G_{n+1})} \alpha^{n+1}_i \big((u,\mathbf{x}),(v, \mathbf{y})\big)= \underbrace{\sum_{(u,v) \in E(G_n)}  \alpha^n_{i}(u,v)}_{=0} - \sum_{\substack{(u,v) \in E(G_n), \\ \mathbf{x}(\beta) \leq \beta(u,v)}} \alpha^n_{i}(u,v).\]
    
    We sum this over $V_{\mathbf{m}}$: \[\texttt{def}_{\alpha^{n+1}_i}(V_{\mathbf{m}}) = \sum_{(u,\mathbf{x}) \in V_{\mathbf{m}}} -\left( \sum_{(u,v) \in E(G_n), \ \mathbf{x}(\beta) \leq \beta(u,v)} \alpha^n_{i}(u,v) \right) = - \sum_{(u,v) \in E(G_n)} \sum_{\mathbf{x}:\mathbf{x}(\beta) \leq \beta(u,v)} \alpha^n_{i}(u,v).\]
    
    Note that the coordinates of $\mathbf{x}$ can take $d^2D_nM_n$ possible values. In the $\beta$ coordinate $d^2D_n\beta(u,v)$ choices result in $\mathbf{x}(\beta) \leq \beta(u,v)$, and we can choose all other coordinates arbitrarily. Consequently,
    
\begin{align*}
    \texttt{def}_{\alpha^{n+1}_i}(V_{\mathbf{m}}) &= -\sum_{(u,v) \in E(G_n)} d^2D_n\beta(u,v)(d^2D_nM_n)^{|\mathcal{B}_n|-1} \alpha^n_i(u,v) \\[6pt] &= -d^2D_n(d^2D_nM_n)^{|\mathcal{B}_n|-1} |E(G_n)| \langle\alpha^n_i,\beta\rangle =0.        
\end{align*}

\end{proof}

\begin{remark}
If $\mathbf{m}$ has a single tight coordinate then the graph $V_{\mathbf{m}}$ consists of $d^{2|\mathcal{B}_n|}$ isomorphic copies of the induced subgraph on $V_{\mathbf{m}} \cap \underline{V}_0(G_{n+1})$, and every $\alpha_i^{n+1}$ is invariant under the isomorphisms between these copies. The deficiency is zero for all of these copies.
\end{remark}

\subsubsection{Surgery and extension.} \label{subsubsec:surgery_and_extension}

The goal of this subsection is to prove the following lemma, that takes care of the construction of $G_{n+1}$.

\begin{lemma}[Extension lemma] \label{lemma:extension}
    If $K_n$ is large enough, we can add vertices $U(G_{n+1})$ and edges such that \begin{enumerate}
        \item \label{itm:proportion} the proportion of new vertices $\frac{|U(G_{n+1})|}{|V(G_{n+1})|}$ is less than $\theta_n$;
        
        \item \label{itm:circ} the $\alpha^{n+1}_i$ already defined on $E(V_0)$ can be extended to circulations with values bounded by $M$ and denominator dividing $D$ such that they remain almost orthogonal: 
        
        $\langle \alpha_i^{n+1}, \alpha_j^{n+1} \rangle \leq \frac{1}{2k}-\frac{1}{2k+2n}$ for $i \neq j$; 
        
        \item \label{itm:bounded_potential_gradient} for each $\beta \in \mathcal{B}_n$ the corresponding coordinate function $p^n_{\beta}$ can be extended to a potential $h_{\beta}$ on $V(G_{n+1})$ such that $\|h_{\beta_{n}}\|_{\infty}\leq K_n M_n$ and $\| \nabla h_{\beta} \|_{\infty} \leq M_n$;
        
        \item{$p_0^n$} \label{itm:proper} extends to a homomorphism $f_n: G_{n+1} \to G_n$ such that the preimage of every edge in $G_n$ has the same size.

        \item The graph $G_{n+1}$ is connected and non-bipartite.
    \end{enumerate}
\end{lemma}
\begin{proof}

First, let us reattach all the edges on the boundary of $V_0(G_{n+1})$ that were cut because they leave the $K_n M_n$-by-$K_nM_n$ box, but now we add distinct endvertices for them. That is, for $e=vw \in E(G_n)$ and $\mathbf{x} \in ([K_n M_n]_{D_n})^{\mathcal{B}_n}$ such that $\mathbf{x}(\beta) + \beta(vw) \notin [K_n M_n]_{D_n}$ for at least one coordinate $\beta$, we add a vertex $u_{e,\mathbf{x}}$ and a $G_{n+1}$-edge $\tilde{e}$ between $(v,\mathbf{x})$ and $u_{e,\mathbf{x}}$, and define $\alpha^{n+1}_i(\tilde{e}) = \alpha^n_i(e)$. Now the flow condition is satisfied at all vertices in $V_0(G_{n+1})$, and all new vertices $u_{e,\mathbf{x}}$ have degree 1. We partition the newly introduced vertices according to the $M_n$-by-$M_n$ cube that their unique neighbor belongs to: $U_{0,\mathbf{m}}=\{u_{e,\mathbf{x}} \mid (v,\mathbf{x}) \in V_{\mathbf{m}}\}$.

Let us explain in this paragraph why almost orthogonality will follow from the construction.
We refer to the set of edges introduced here as the \emph{boundary} of $V_0(G_{n+1})$, denoted by $\partial(V_0)$. We partition $\partial(V_0)$ into $\partial^{+}(V_0)$ and $\partial^{-}(V_0)$ as follows. Fix any order on $\mathcal{B}_n$, and for each $\tilde{e}$ let $\beta_{\tilde{e}}$ denote the first coordinate $\beta$ such that $\mathbf{x}(\beta) + \beta(vw) \notin [K_n M_n]_{D_n}$. If $\mathbf{x}(\beta_{\tilde{e}}) + \beta_{\tilde{e}}(vw) \leq 0$, then $\tilde{e}$ belongs to $\partial^{-}(V_0)$, and if $\mathbf{x}(\beta_{\tilde{e}}) + \beta_{\tilde{e}}(vw) > K_n \cdot M_n$, then $\tilde{e}$ belongs to $\partial^{-}(V_0)$. The point is that the original edges $E(V_0)$ together with, say $\partial^{+}(V_0)$ provide a perfect cover of $G_n$, and the $\alpha^{n+1}_i$ are the lift of the $\alpha^{n}_i$ on these edges. So the sum of $\alpha^{n+1}_i\cdot \alpha^{n+1}_j$ over $E(V_0)\cup\partial^{+}(V_0)$ is close to zero. Thus, the almost orthogonality will be ensured as long as the proportion of newly added edges is small. 

The next two claims perform different surgeries on the boundary boxes. We work with boxes of the form $\underline{V}_{\mathbf{m}}=V_{\mathbf{m}} \cap \underline{V}_0({G_{n+1}})$ and their isomorphic copies. The surgery will add path-like gadgets in order to connect vertices with deficit to vertices with sufficit. First, we consider the case when $\mathbf{m}$ has one single tight coordinate. By the Deficiency Lemma \ref{lemma:deficiency} we can match vertices with deficit to those with sufficit within the same box, hence these paths do not need to be long. Hence, the number of additional vertices will be proportional to the total size of these boxes, that is, $O(1/K_n)|V_0(G_{n+1})|$. We will need to be more careful in the proof of Claim \ref{lemma:surgery2}, when we simultaneously treat all the boxes with at least two tight coordinates. We have to allow these paths to have length $O(K_n)$ in order to extend every $p_{\beta}^n$ to a potential with bounded gradient (i.e., not depending on $K_n$). This will be compensated by the fact that the proportion of these boxes is $O(1/K_n^2)$, so the number of added vertices is again $O(1/K_n)|V_0(G_{n+1})|$.

\begin{claim} \label{lemma:surgery}
If $\mathbf{m}$ has one single tight coordinate, then we can add a set of vertices $U_{\mathbf{m}}$ to $V_0({G_n+1})$ and a set of edges with at least one endvertex in $U_{\mathbf{m}}$ such that 

\begin{enumerate}

\item \label{egy}
$|U_{\mathbf{m}}| \leq c|\underline{V}_{\mathbf{m}}|$, where $c$ depends only on $d, k, DM,$ and $G_n$, but not on $K_n$,

\item 
the vertices of $U_{\mathbf{m}}$ are not adjacent to any vertex of 
$V_0(G_{n+1}) \setminus \underline{V}_{\mathbf{m}}$,

\item 
the resulting graph is $d$-regular at every vertex of $\underline{V}_{\mathbf{m}} \cup U_{\mathbf{m}}$,

\item \label{negy}
the $\alpha^{n+1}_i$ extend to $U_{\mathbf{m}}$ such that the flow condition holds at every vertex of $\underline{V}_{\mathbf{m}} \cup U_{\mathbf{m}}$,

\item \label{ot}
$|\alpha^{n+1}_i| \leq M$ on the set of edges spanned by $\underline{V}_{\mathbf{m}} \cup U_{\mathbf{m}}$,

\item
the homomorphism $p_0^n$ extends to a homomorphism $f_n$ on $U_{\mathbf{m}} \cup V_0(G_{n+1})$.

\end{enumerate} 
\end{claim}

\begin{proof}
We have already added $U_{0,\mathbf{m}}$. From each vertex $u\in U_{0,\mathbf{m}}$, build a $(d-1)$-ary tree of depth $\left\lceil \log_{d-1}\big(k(MD)^k +2 \big)\right\rceil$. At the root $u$, each $\alpha^{n+1}_{i}$ defines an inflow, which we will spread out over the tree. We define the $\alpha_i^{n+1}$ as we go down the tree, splitting the values, and eventually subdivide the  vector $(\alpha_1^{n+1}, \ldots, \alpha^{n+1}_k)$ into parts with at most one nonzero coordinate taking value $\pm \frac{1}{D}$ and multiple vertices with zero coordinates only. We do this for every $u$. By Lemma~\ref{lemma:deficiency}, the number of leaves with value $\frac{1}{D}$ and $-\frac{1}{D}$ in the $i$-th coordinate are the same, for every $i$. Therefore, we can choose a matching on the leaves of the trees connecting the $\frac{1}{D}$ values to the $-\frac{1}{D}$ values. We connect every matched pair of leaves by a gadget graph. The rest of the leaves can be matched since $M_n$ and hence the number of leaves is even. We only make sure that there are two leaves in the same tree matched to each other in order to create an odd cycle, hence the resulting graph will not be bipartite. For the natural number $m$, the gadget graph $\texttt{Gad}(m)$ will consist of $m$ consecutive copies of the complete bipartite graph $K_{d-1,d-1}$, joined by $d$ edges bijectively between the appropriate bipartition classes, and connected to two distinguished vertices at the end. Formally, we set $V(\texttt{Gad}(m))=\{u,v\} \cup [2m] \times [d-1]$, and define the set of edges as follows. 
\begin{align*}
    E(\texttt{Gad}(m))=&\bigg\{\big(u,(1,j) \big), \big( v,(2m,j) \big): j \in [d-1] \bigg\} \\
&\cup \bigg\{\big( (2i,j),(2i+1,j) \big): i \in [m-1], j \in [d-1] \bigg\}\\
 &\cup\bigg\{\big( (2i-1,j),(2i,j') \big): i \in [m], j, j' \in [d-1] \bigg\}.
 \end{align*}

For every matched pair $(a,b) $ we add a copy of $\texttt{Gad}(m(a,b))$ to the graph, where $m(a,b)\leq |V(G_n)|$ will be chosen later, and identify $a$ with $u$ and $b$ with $v$. These gadgets will be pairwise disjoint. We send the flow along a path from $a$ to $b$ in the gadget, and define it to be zero on the other edges. Now $\ref{egy}-\ref{ot}$  hold for this construction. 

It remains to extend $p_0^n$ to a homomorphism $f_n$. We can define $f_n$ arbitrarily on the trees rooted at the vertices of $U_{0,\mathbf{m}}$ as long as we map each child to a neighbor of the image of its ancestor. The gadget $\texttt{Gad}(m)$ has a homomorphism into the path of length $2m+1$ mapping $u$ and $v$ to the endvertices. Therefore, we can choose $m(a,b)$ for a matched pair $(a,b)$ such that there is a walk of length $2m(a,b)+1$ connecting $f_{n}(a)$ and $f_n(b)$. Such a path exists, since $G_n$ is connected and not bipartite.
\end{proof}

Now we treat all the other boxes on the boundary together. Let $V_0^+$ denote the union of the boxes $V_{\mathbf{m}}$ with at least two tight coordinates. Set $\underline{V}_0^+=V_0^+ \cap \underline{V}_0(G_{n+1})$.

\begin{claim}\label{lemma:surgery2}
We can add a set of vertices $U_0^+$ to $V_0(G_{n+1})$ and a set of edges with at least one endvertex in $U_0^+$ such that 

\begin{enumerate}

\item 
$|U_0^+| \leq c K_n |\underline{V}_0^+|$, where $c$ depends only on $d, k, DM, G_n$ and $\mathcal{B}_n$, but not on $K_n$,

\item 
the vertices of $U_0^+$ are not adjacent to any vertex of 
$V_0(G_{n+1}) \setminus \underline{V}_0^+$,

\item 
the resulting graph is $d$-regular at every vertex of $\underline{V}_0^+ \cup U_0^+$,

\item
the lift of each $\alpha^n_i$ extends to $U_0^+$ such that the flow condition holds at every vertex of  $\underline{V}_0^+ \cup U_0^+$,

\item
$|\alpha^{n+1}_i| \leq M$ on the set of edges spanned by $\underline{V}_0^+ \cup U_0^+$,

\item
the homomorphism $p_0^n$ extends to a homomorphism on  $\underline{V}_0^+ \cup U_0^+$,

\item \label{extra} for each $\beta \in \mathcal{B}_n$ the corresponding coordinate function $p^n_{\beta}$ can be extended to a potential $h_{\beta}$ on $\underline{V}_0^+ \cup U_0^+$ such that $\|h_{\beta_{n}}\|_{\infty}\leq K_n M_n$ and $\| \nabla h_{\beta} \|_{\infty} \leq M_n$.
\end{enumerate} 
\end{claim}

\begin{proof}
Note that the total deficiency of any $\alpha_i^{n+1}$ on $\underline{V}_0^+$ is again zero. The construction will be similar to that in the previous lemma.  The difference is that in order to satisfy \ref{extra}, we choose $m(a,b)>K_n$. 
This allows us to extend each $p^n_{\beta}$ to every gadget as a linear function of the distance from $a$ such that \ref{extra} holds, since the distance of $a$ and $b$ is at least $K_n$ inside the gadget.
The rest of the proof is the same. 
\end{proof}

We apply Claim \ref{lemma:surgery} to the boxes with exactly one tight coordinate $\mathbf{m}$, consisting of disjoint isomorphic copies of $\underline{V}_{\mathbf{m}}$, and Claim \ref{lemma:surgery2} to the union of the boxes with at least two tight coordinates, consisting of disjoint isomorphic copies of $\underline{V}_0^+$. These allow us to extend each $\alpha_i^{n+1}$ and the homomorphism $p_0^n$ from $V_0(G_{n+1})$, while each $p^n_{\beta}$ is extended to $h_\beta$ on $U_0^+$. We can define $h_\beta$ as a constant on $U_{\mathbf{m}}$ for every $\mathbf{m}$ with one single tight coordinate. 

{\bf Connectivity of $G_{n+1}$.} Now we apply some surgeries to make $G_{n+1}$ connected. After some preparation we will choose pairs of edges $(a,b)$ and $(a',b')$ in different components, whose removal does not disconnect their component, remove them and add the edges $(a,b')$ and $(a',b)$ in order to merge these components. Before making this surgery we will make sure that the circulations $\alpha_i^{n+1}$ take the same values on these edges, that $(a,b$) and $(a',b')$ are mapped to the same edge of $G_n$, and the values of the potentials $h_{\beta}$ almost agree at $a$ and $a'$, and also at $b$ and $b'$. The first two conditions enable us to maintain $\alpha_i^{n+1}$ and $f_n$ during the surgery, whereas the last one will guarantee that $\|\nabla h_{\beta}\|_{\infty} \leq M_n$ still holds. Hence, in preparation, we will replace certain edges $(x,y)$ by a gadget $\texttt{Gad}(m)$, i.e., we remove the edge $(x,y)$, add an isomorphic copy of $\texttt{Gad}(m)$, and connect $x$ to $u$ and $y$ to $v$. We fix an edge $e$ of $G_n$, and choose the size of the gadget $\texttt{Gad}(m)$ such that there is a walk of length $(2m+3)$ connecting $f_n(x)$ and $f_n(y)$ in $G_n$ going through $e$ at least $2|\mathcal{B}_n|$ times, and define the homomorphism $f_n$ on the gadget accordingly. We extend each $h_{\beta}$ to the gadget linearly, hence $\nabla h_{\beta}$ is at most $M_n/(2m+3)$ on the gadget edges. We do this replacement for every edge in $\partial V_0(G_{n+1})$. We can choose $2|\mathcal{B}_n|$ preimages of $e$ in each gadget introduced above, such that these preimages correspond to different steps of the walk on $G_n$. Removing these edges simultaneously does not disconnect the gadget. We associate two such edges in every gadget to each element of $\mathcal{B}_n$: one to increase that coordinate, and one to decrease it.   

Given two boundary edges whose endvertices differ by exactly $\frac{1}{d^2D_n}$ in exactly one coordinate $\beta \in \mathcal{B}_n$, we choose the edges $(a,b), (a',b')$ associated to $\beta$ from their respective gadgets, and perform the cross-surgery described above. (For one of the boundary edges we choose $(a,b)$ to be the gadget-edge that is meant to decrease the $\beta$ coordinate, and for the other edge, we choose $(a',b')$ to be the gadget-edge meant to increase it, depending on which edge had endvertices with higher $\beta$-coordinate.) The value of $h_{\beta}$ differs on $a$ and $a'$ (and also on $b$ and $b'$) by $\frac{1}{d^2D_n}$, so $\nabla h_{\beta}$ can increase by at most this amount on the edges, implying $\nabla h_{\beta} (a,b') \leq \frac{M_n}{2m+3} + \frac{1}{d^2D_n} \leq M_n$, and similarly for $(a',b)$.

Why is this graph connected? First, note that two vertices on the boundary of $V_0(G_{n+1})$ can be connected if they agree in their coordinate corresponding to $V(G_n)$ and have the same tight coordinate, since every coordinate can be changed by $\frac{1}{d^2D_n}$. (Note that the tight coordinate can also be changed, as long as the result of the change is still a vertex on the boundary of $V_0(G_{n+1})$.) Next, we use our technical assumption on $B_n$. Namely, for every $x \in V(G_n)$ there are at least two elements of $\mathcal{B}_n$ that are not zero on edges incident to $x$. Hence, the faces of the boundary corresponding to these coordinates contain vertices whose first coordinate is $x$, moreover, there is such a vertex which is contained in both faces. Consequently, vertices on the boundary of $V_0(G_{n+1})$, whose first coordinates agree are in the same component. 

Now consider two adjacent vertices $u,v \in V(G_n)$. We claim that there are two adjacent vertices $(u, \mathbf{x}),(v,\mathbf{y})$ on the boundary of $V_0(G_{n+1})$. Consequently, all the vertices on the boundary of $V_0(G_{n+1})$ are in the same connected component.   
To find such vertices $(u, \mathbf{x})$ and $(v,\mathbf{y})$, we use again our technical assumption on $\mathcal{B}_n$. For each of $u$ and $v$ there are two distinct elements of $\mathcal{B}_n$, where vertices with first coordinate $x$ and $y$, respectively, appear on both corresponding faces of the boundary. Hence, we can choose two distinct coordinates, $\beta$ and $\beta'$, that have this property for $u$ and $v$, respectively. We will choose $\mathbf{x}(\beta)$ and $\mathbf{y}(\beta')$ such that they respectively ensure that $(u, \mathbf{x})$ and $(v,\mathbf{y})$ are on the boundary of $V_0(G_{n+1})$. We have to be a bit careful: as we want $(u, \mathbf{x})$ and $(v,\mathbf{y})$ to be connected, choosing $\mathbf{x}(\beta)$ prescribes $\mathbf{y}(\beta)$ via (\ref{eqn:edge_lift}). We assumed that the two faces corresponding to $\beta$ both contain vertices with first coordinate $u$. That is, we can choose $\mathbf{x}(\beta)$ sufficiently close to either $0$ or $K_nM_n$ to be on the boundary, and one of these two choices, depending on the sign of $\beta(u,v)$, will result in $0 < \mathbf{y}(\beta) \leq K_nM_n$. We do the same when choosing $\mathbf{y}(\beta')$: we ensure that $(v,\mathbf{y})$ is on the boundary, and $0<\mathbf{x}(\beta')\leq K_nM_n$. We are free to choose the rest of the coordinates of $\mathbf{x}$ and $\mathbf{y}$ arbitrarily as long as their difference satisfies (\ref{eqn:edge_lift}) and all the values are in $[K_nM_n]_{d^2D_n}$.

Finally, we argue that all vertices are connected to the boundary. Indeed, given an arbitrary vertex $(u,\mathbf{x})$ of $V_0(G_{n+1})$, we choose a cycle in $G_n$ containing $u$ such that some $\beta \in \mathcal{B}_n$ does not sum to zero on the cycle. Iteratively following the lifts of this cycle will keep increasing or decreasing the $\beta$ coordinate, and we will eventually hit the boundary. This finishes the proof of the connectivity of $G_{n+1}$.

{\bf Making the sequence proper.} In order to guarantee that the preimage of every edge in $G_n$ has the same size, we iteratively replace edges of $G_{n+1}$ whose image has a smaller preimage than the maximum over $E(G_n)$ by the above gadget $\texttt{Gad}(1)$. We might map all the new edges of $\texttt{Gad}(1)$ to the same edge of $G_n$ as the old one. We increase the size of the preimage of an edge by $d^2$ each step. The size of every preimage is divisible by $d^2$, since we made the same surgery in every copy of $\underline{V}_0(G_{n+1})$, and cross-surgeries also preserve this property. Thus, \ref{itm:proper} will hold in the end of this process. And the size of the largest preimage has not changed, hence the number of vertices not in $V_0(G_{n+1})$ can become at most $|E(G_n)|$ times bigger during this procedure.
Every potential can be extended to these gadgets in order to satisfy \ref{itm:bounded_potential_gradient}.

{\bf Bounding the size of the boundary.} Note that $|V_0^+| \leq {|\mathcal{B}_n| \choose 2} \frac{4}{K_n^2}|V_0(G_{n+1})|$, while the union of the boxes with one single tight coordinate has size at most 
$|\mathcal{B}_n| \frac{2}{K_n}|V_0(G_{n+1})|$.
Consequently, the number of vertices added using Claims~\ref{lemma:surgery} and \ref{lemma:surgery2}, the cross-surgery, and the surgeries making the size of the preimages equal  is  $O(\frac{1}{K_n})|V_0(G_{n+1})|$. Hence \ref{itm:proportion} follows if $K_n$ is large enough.

{\bf Almost orthogonality.} We prove almost orthogonality by induction. The contribution of the original edges $E_0(G_{n+1})$ together with $\partial^+(V_0)$ is bounded by the induction. The circulations take value at most $M$ on the rest of the edges, whose proportion is bounded by $\theta_n$. So we will be able to make sure that the contribution of the latter is at most $\frac{1}{2k}-\frac{1}{2k+2n} - \big(\frac{1}{2k}-\frac{1}{2k+2n-2}\big)=\frac{1}{2(k+n)(k+n-1)}>0$. This holds if $\theta_n \leq \frac{1}{2(k+n)(k+n-1)M^2}$. 
\end{proof}

\subsubsection{$\mcT$ as an inverse limit}

We define $\mcT$ as the inverse limit of $\{ G_n \}_{n=1}^{\infty}$ equipped with the homomorphisms $f_n:V(G_{n+1}) \rightarrow V(G_n)$. Since this is a $d$-proper sequence
Lemma \ref{basic} shows that the limit is a $d$-regular graphing.

We denote the normalized edge measure by $\eta$. Observe that $\alpha^n_i(e)=\alpha^{n+1}_i(e)$ if both endvertices of $e$ are in $V_0(G_{n+1})$. As long as we choose $\theta_n$ to be small enough, we have $\sum_{n=1}^{\infty} \frac{|U(G_n)|}{V(G_n)}<\infty$, and the sequence $\alpha^n_i(e)$ stabilizes for a.e. $e \in E(\mcT)$ by the Borel-Cantelli lemma. Thus, $\alpha_i(e)=\lim_{n\to \infty}\alpha_i^{n}(e)$ are a.e.\ well-defined and are circulations on $\mathcal{T}$. 

\subsection{The elimination of the other circulations} \label{subsec:no_other_circ}

We will work in $L^2(E(\mcT), \eta)$.
We denote the $L^2$ norm by $\|*\|$ and the supremum norm by $\|*\|_{\infty}$. Given $m$ and a mapping $\varphi:E(G_m) \to \mathbb{R}$, it induces a map$\underline{\varphi}: E(\mcT) \to \mathbb{R}$ defined by $\underline{\varphi}(e)=\varphi(e_m)$. In what follows we omit the underline, e.g., by $\|\varphi\|$ we mean $\|\underline{\varphi}\|$, this will be clear from the context. The same applies to $\|*\|_{\infty}$, and also to $L^2(V(\mcT),\lambda)$ w.r.t.\ both norms. 

\begin{lemma} \label{lemma:alpha_norms}
The norms of the $\alpha_i$ and $\alpha^n_i$ are close to one: for every $i$

$\Pi_{j=1}^{n-1} (1-\theta_j) \leq \|\alpha_i^n\|^2 \leq 1+ \sum_{j=1}^{n-1} \theta_j M^2$, and

$\Pi_{j=1}^{\infty} (1-\theta_j) \leq \|\alpha_i\|^2 \leq 1 + \sum_{j=1}^{\infty} \theta_jM^2$.
\end{lemma}

\begin{proof} We prove the bounds by induction on $n$. Recall that $\|\alpha^1_i\|^2=1$.

Note that $(1-\theta_n) \|\alpha_i^n\|^2 \leq \frac{|V_0(G_{n+1})|}{|V(G_{n+1})|} \|\alpha_i^n\|^2 = \|\alpha_i^{n+1}|_{V_0(G_{n+1}) \times V(G_{n+1})}\|^2 \leq \|\alpha_i^{n+1}\|^2$. 

On the other hand, 
\begin{align*}
    \|\alpha_i^{n+1}\|^2 &= \|\alpha_i^{n+1}|_{V_0(G_{n+1}) \times V(G_{n+1})}\|^2 + \|\alpha_i^{n+1}|_{E(U(G_{n+1})) \times V(G_{n+1})}\|^2 \\[6pt] &\leq \frac{|V_0(G_{n+1})|}{|V(G_{n+1})|}\|\alpha_i^n\|^2 + \frac{|U(G_{n+1})|}{|V(G_{n+1})|} M^2 
\leq \|\alpha_i^n\|^2 + \theta_n M^2.
\end{align*}

Since $\alpha_i^n$ converges to $\alpha_i$ in $L^2(\mcT, \eta)$ the second part of the lemma follows as well.
\end{proof}

Now we prove our main theorem, namely that the $\alpha_i$'s span the vectorspace of all $L^2$ circulations of $\mcT$.

\begin{proof}[Proof of Theorem~\ref{thm:finitely_many_circ}] 
      If a graphing contains cycles on a set of vertices of positive measure, one can easily define infinitely many linearly independent circulations in $L^2$. Therefore, showing that the $\alpha_i$ generate all the circulations implies that $\mcT$ is a treeing.
      
      Suppose for a contradiction that $\gamma \in L^2\big(E(\mcT), \eta\big)$ is a circulation with $\gamma \perp \alpha_i$ for all $i$ and $\|\gamma\|=1$. As mentioned in section \ref{subsec:strategy}, we will build a potential $h$ on $\mcT$ such that $\langle\gamma, \nabla h\rangle$ will be close to $1$, which contradicts the fact that it should be zero by Lemma~\ref{lem:potential_vs_circulation}. We illustrate the argument in Figure~\ref{fig:picture_of proof}. 

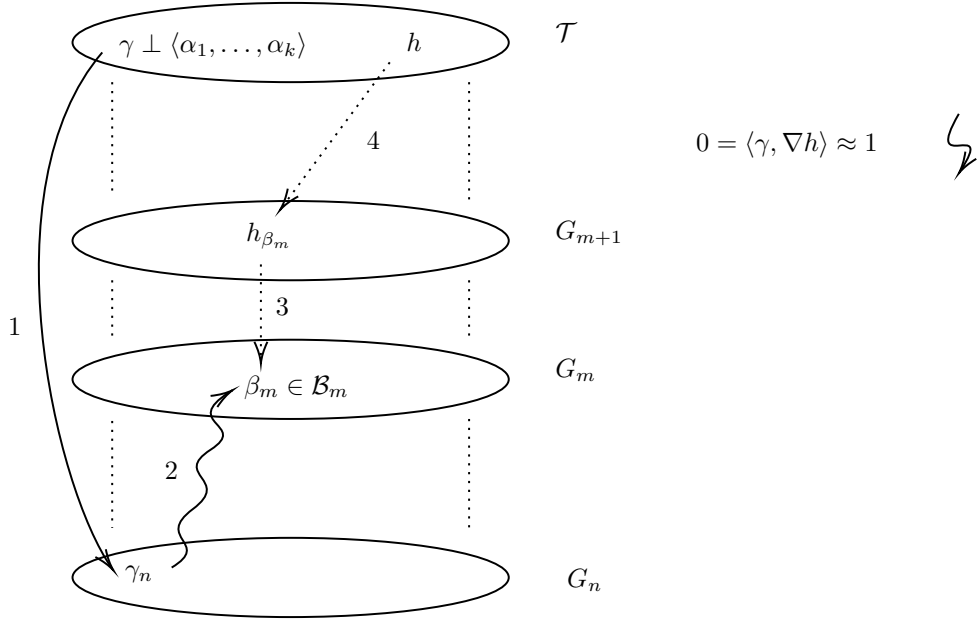
\begin{figure}
    \centering

\tikzset{every picture/.style={line width=0.75pt}} 

\begin{tikzpicture}[x=0.75pt,y=0.75pt,yscale=-1,xscale=1]

\draw   (80,70) .. controls (80,58.95) and (129.25,50) .. (190,50) .. controls (250.75,50) and (300,58.95) .. (300,70) .. controls (300,81.05) and (250.75,90) .. (190,90) .. controls (129.25,90) and (80,81.05) .. (80,70) -- cycle ;
\draw [line width=0.75]  [dash pattern={on 0.84pt off 2.51pt}]  (100,90) -- (100,145) ;
\draw [line width=0.75]  [dash pattern={on 0.84pt off 2.51pt}]  (280,90) -- (280,150) ;
\draw [line width=0.75]  [dash pattern={on 0.84pt off 2.51pt}]  (100,190) -- (100,220) ;
\draw [line width=0.75]  [dash pattern={on 0.84pt off 2.51pt}]  (280,190) -- (280,220) ;
\draw    (95,75) .. controls (36.26,143.81) and (70.96,294.45) .. (99.79,335.3) ;
\draw [shift={(100.67,336.5)}, rotate = 233.37] [color={rgb, 255:red, 0; green, 0; blue, 0 }  ][line width=0.75]    (10.93,-3.29) .. controls (6.95,-1.4) and (3.31,-0.3) .. (0,0) .. controls (3.31,0.3) and (6.95,1.4) .. (10.93,3.29)   ;
\draw  [dash pattern={on 0.84pt off 2.51pt}]  (175,182) -- (175,230) ;
\draw [shift={(175,232)}, rotate = 270] [color={rgb, 255:red, 0; green, 0; blue, 0 }  ][line width=0.75]    (10.93,-3.29) .. controls (6.95,-1.4) and (3.31,-0.3) .. (0,0) .. controls (3.31,0.3) and (6.95,1.4) .. (10.93,3.29)   ;
\draw  [dash pattern={on 0.84pt off 2.51pt}]  (240,80) -- (185.86,152.89) ;
\draw [shift={(184.67,154.5)}, rotate = 306.6] [color={rgb, 255:red, 0; green, 0; blue, 0 }  ][line width=0.75]    (10.93,-3.29) .. controls (6.95,-1.4) and (3.31,-0.3) .. (0,0) .. controls (3.31,0.3) and (6.95,1.4) .. (10.93,3.29)   ;
\draw [line width=0.75]  [dash pattern={on 0.84pt off 2.51pt}]  (100,261) -- (100,315) ;
\draw [line width=0.75]  [dash pattern={on 0.84pt off 2.51pt}]  (280,260) -- (280,315) ;
\draw    (527,106) .. controls (508.29,138.34) and (545.84,106.81) .. (527.86,134.68) ;
\draw [shift={(527,136)}, rotate = 303.54] [color={rgb, 255:red, 0; green, 0; blue, 0 }  ][line width=0.75]    (10.93,-3.29) .. controls (6.95,-1.4) and (3.31,-0.3) .. (0,0) .. controls (3.31,0.3) and (6.95,1.4) .. (10.93,3.29)   ;
\draw   (80,170) .. controls (80,158.95) and (129.25,150) .. (190,150) .. controls (250.75,150) and (300,158.95) .. (300,170) .. controls (300,181.05) and (250.75,190) .. (190,190) .. controls (129.25,190) and (80,181.05) .. (80,170) -- cycle ;
\draw   (80,240) .. controls (80,228.95) and (129.25,220) .. (190,220) .. controls (250.75,220) and (300,228.95) .. (300,240) .. controls (300,251.05) and (250.75,260) .. (190,260) .. controls (129.25,260) and (80,251.05) .. (80,240) -- cycle ;
\draw   (80,340) .. controls (80,328.95) and (129.25,320) .. (190,320) .. controls (250.75,320) and (300,328.95) .. (300,340) .. controls (300,351.05) and (250.75,360) .. (190,360) .. controls (129.25,360) and (80,351.05) .. (80,340) -- cycle ;
\draw    (130,335) .. controls (150.67,322.5) and (120.67,320.5) .. (140.67,305.5) .. controls (160.67,290.5) and (129.67,288.5) .. (149.67,275.5) .. controls (169.27,262.76) and (134.12,262.51) .. (159.06,246.5) ;
\draw [shift={(160.67,245.5)}, rotate = 148.74] [color={rgb, 255:red, 0; green, 0; blue, 0 }  ][line width=0.75]    (10.93,-3.29) .. controls (6.95,-1.4) and (3.31,-0.3) .. (0,0) .. controls (3.31,0.3) and (6.95,1.4) .. (10.93,3.29)   ;

\draw (322,58.4) node [anchor=north west][inner sep=0.75pt]    {$\mathcal{T}$};
\draw (322,158.4) node [anchor=north west][inner sep=0.75pt]    {$G_{m+1}$};
\draw (322,228.4) node [anchor=north west][inner sep=0.75pt]    {$G_{m}$};
\draw (102,63.4) node [anchor=north west][inner sep=0.75pt]    {$\gamma \perp \langle \alpha _{1} ,\dotsc ,\alpha _{k} \rangle $};
\draw (105,332.4) node [anchor=north west][inner sep=0.75pt]    {$\gamma _{n}$};
\draw (165,236.4) node [anchor=north west][inner sep=0.75pt]    {$\beta _{m} \in \mathcal{B}_{m}$};
\draw (166,159.4) node [anchor=north west][inner sep=0.75pt]    {$h_{\beta _{m}}$};
\draw (247,63.4) node [anchor=north west][inner sep=0.75pt]    {$h$};
\draw (393,113.4) node [anchor=north west][inner sep=0.75pt]    {$0=\langle \gamma ,\nabla h\rangle \approx 1$};
\draw (46,207.4) node [anchor=north west][inner sep=0.75pt]    {$1$};
\draw (125,280.4) node [anchor=north west][inner sep=0.75pt]    {$2$};
\draw (181,197.4) node [anchor=north west][inner sep=0.75pt]    {$3$};
\draw (227,113.4) node [anchor=north west][inner sep=0.75pt]    {$4$};
\draw (328,334.4) node [anchor=north west][inner sep=0.75pt]    {$G_{n}$};

\end{tikzpicture}

    \caption{ Outline of the proof of Theorem~\ref{thm:finitely_many_circ}}
    \label{fig:picture_of proof}
\end{figure}

First, we approximate $\gamma$ by a function $\gamma_n$ (its conditional expected value) w.r.t. the finite $\sigma$-algebra generated by the edges of $G_n$:
\[\gamma_n(e) = |E(G_n)| \int_{\tilde{e}_n = e_n} \gamma(\tilde{e}) \ d \eta(\tilde{e}).\] 
The sequence $\gamma_n$ converges to $\gamma$ a.e.\ and, since the norms admit a uniform bound, in $L_2(E(\mcT), \eta)$ as well.
Thus, if $n$ is large enough then $\|\gamma- \gamma_n \| < \varepsilon$ and $1-\varepsilon \leq \|\gamma_n\| \leq 1$. We get the following lemma similarly.

\begin{lemma}[Finite approximations $\gamma_n$ are almost orthogonal to the $\alpha^n_i$] \label{lemma:almost_orthogonal}
    If $n$ is large enough, then $|\langle \gamma_n, \alpha^n_i \rangle | \leq \frac{\varepsilon}{4k}$.
\end{lemma}

\begin{proof}
The sequence $\gamma_n$ converges to $\gamma$ in $L^2$, while $\alpha^n_i$ converges to $\alpha_i$ in $L^2$. Thus, $\lim_{n \to \infty} \langle \gamma_n, \alpha^n_i \rangle = \langle \gamma, \alpha_i \rangle=0$, and the lemma follows.   
\end{proof}

A priori $\gamma_n$ might have very large entries. Nevertheless, by possibly passing to some $m>n$ large enough we can assume that $\|\gamma_n\|_{\infty} < M_m - \frac{\varepsilon}{2} M$  
and $|\langle \gamma_n, \alpha^m_i \rangle | \leq \frac{\varepsilon}{4k}$ still hold for every $i$. By definition,  $\gamma_n$ assigns values to edges of $\mathcal{T}$. Yet it is constant on preimages of edges of $E(G_n)$, so we can think of it also as a function on $E(G_n)$, and consequently also on $E(G_m)$. We believe these interpretations are clear from context, so we will not burden the notation to indicate them.

\begin{lemma}
    There exists $\beta_m \in \mathcal{B}_m$ such that $\|\beta_m - \gamma_n\| <\varepsilon$.
\end{lemma}
\begin{proof}

We denote by $\gamma^{\perp}_n$ the component of $\gamma_n$ perpendicular to all the $\alpha^m_i$ on $G_m$. As $\gamma_n$ is almost orthogonal to the $\alpha^m_i$ by Lemma~\ref{lemma:almost_orthogonal}, $\gamma^{\perp}_n$ is close to $\gamma_n$. Formally, consider the matrix $A \in \mathbb{R}^{k \times k}$ with entries $A_{i,j}=\langle \alpha_i, \alpha_j \rangle$, the vector $b \in \mathbb{R}^k$ defined as $b_i=\langle \gamma_n,\alpha_i^m \rangle$ and the solution $x \in \mathbb{R}^k$ of $Ax=b$.

\begin{equation} \label{eqn:gamma_prime}
\gamma^{\perp}_n=\gamma_n-\sum_{i=1}^k x_i \alpha_i^m
\end{equation}

We use that the norms of the $\alpha_i^m$ are almost one by Lemma~\ref{lemma:alpha_norms} if $\theta_n$ are small enough: $1 - \frac{1}{2k} < \|\alpha_i\|^2 < 1 + \frac{1}{2k}$, while $|\langle \alpha_i, \alpha_j \rangle | \leq \frac{1}{2k}$ for every $i \neq j$. Thus, $A=I+A_0$, where $I$ is the identity matrix and every entry of the matrix $A_0$ has absolute value at most $\frac{1}{2k}$, and strictly less than $\frac{1}{2k}$ on the diagonal.
Therefore, $A^{-1}=\sum_{\ell=0}^{\infty}(-A_0)^{-\ell}$, since the sequence is absolutely convergent, and, in absolute value, every entry of $A_0^{\ell}$ is at most $\frac{1}{k2^{\ell}}$, and every entry of $A^{-1}-I$ is at most $\frac{1}{k}$. By Lemma \ref{lemma:almost_orthogonal} every coordinate of $b$ is at most $\frac{\varepsilon}{4k}$, hence the coordinates of $x=A^{-1}b$ are at most $\frac{\varepsilon}{2k}$. 
This shows that the entries of $\gamma^{\perp}_n$ are not much larger in absolute value than those of $\gamma_n$.

    \[\|\gamma^{\perp}_n\|_{\infty} \leq \|\gamma_n| |_{\infty} + \sum_{i=1}^k \frac{\varepsilon}{2k} \|\alpha_i^m\|_{\infty} \leq \|\gamma_n \|_{\infty} + \frac{\varepsilon}{2} \cdot \max_{i} \| \alpha_i^m \|_{\infty} \leq \|\gamma_n \|_{\infty} + \frac{\varepsilon}{2} M \leq  M_m.\]

Thus, $\gamma_n^{\perp} \in \langle\alpha_i^m,\ldots, \alpha_k^m\rangle^{\perp} \cap [-M_m, M_m]^{E(G_m)}$. As $\mathcal{B}_m$ is an $\frac{\varepsilon}{2}$-net, there is a $\beta_m \in \mathcal{B}_m$ such that $\|\gamma^{\perp}_n - \beta_m\| < \frac{\varepsilon}{2}$, and we get the bound in the statement of the lemma by the triangle inequality. 
\end{proof}

Clearly, $|\|\beta_m\|-1| < 2\varepsilon$. According to {\it (\ref{itm:bounded_potential_gradient})} of Lemma~\ref{lemma:extension}, there is an extension of the coordinate function $p^{m}_{\beta_m}$ to a potential $h_{\beta_{m}}$ on $V(G_{m+1})$ such that $\|h_{\beta_{m}}\|_{\infty} \leq K_m M_m$ and $\| \nabla h_{\beta_{m}} \|_{\infty} \leq M_m$. Now $h:V(\mcT) \to \mathbb{R}$ defined by $h(x)=h_{\beta_m}(x_m)$ satisfies $\|h\|_{\infty} \leq K_m M_m$ and $\| \nabla h \|_{\infty} \leq M_m$. Note that on directed edges $e$ such that both endvertices of $e_{m+1}$ are in $V_0(G_{m+1})$ we have $\nabla h(e) = \beta_m(e_m)$, therefore, this holds for at least $1-2d\theta_m$ proportion of the edges in the preimage of every $e_m \in E(G_m)$. 
Consequently, if $\theta_m<\frac{2\varepsilon^2}{d(1-2\varepsilon)^2}$ then

\begin{equation}
\label{eqn:beta_n_inner_prod} \langle \beta_m, \nabla h \rangle \geq (1-2d\theta_m) \|\beta_m\|^2 > (1-2d\theta_m)(1-2\varepsilon)^2 > 1-4\varepsilon.
\end{equation} 

\begin{lemma}[$\nabla h$ has norm almost one.] \label{lemma:nabla_h_norm}
If $\theta_m<\frac{\varepsilon^2}{M_m^2}$ then $\|\nabla h\| < 1+3\varepsilon$.
\end{lemma}
\begin{proof} Since the sequence 
is proper, $\|\nabla h\| < \|\nabla h_{\beta_m}|_{V_0(G_{m+1})}\|+\sqrt{\theta_m} M_m < \|\beta_m\|+\varepsilon < 1+3\varepsilon.$
\end{proof}

We are now ready to show that $\langle\gamma, \nabla h\rangle$ is close to one. We approximate $\gamma$ with $\beta_m$.

\begin{equation} \label{eqn:inner_product_parts}
    \langle \gamma, \nabla h \rangle = \langle \beta_m, \nabla h \rangle + \langle \gamma_n - \beta_m, \nabla h \rangle + \langle \gamma - \gamma_n, \nabla h \rangle.
\end{equation}
We analyze the terms on the right-hand side of equation~\ref{eqn:inner_product_parts}. 

First, using~\ref{eqn:beta_n_inner_prod} we have
\[\langle \beta_m, \nabla h \rangle > 1- 4\varepsilon.\]

Next, \[|\langle \gamma_n - \beta_m, \nabla h \rangle| \leq  \|\gamma_n-\beta_m \| \cdot \|\nabla h\| < \varepsilon (1+3\varepsilon).\]

Finally, 
\[|\langle \gamma - \gamma_n, \nabla h \rangle| \leq \|\gamma-\gamma_n \| \cdot \|\nabla h\| < \varepsilon(1+3\varepsilon).\]

Plugging these three bounds into \ref{eqn:inner_product_parts} we get
\[\langle \gamma, \nabla h \rangle > (1-4\varepsilon)- \varepsilon(1+3\varepsilon)-\varepsilon(1+3\varepsilon)=1-6\varepsilon -6\varepsilon^2>0,\] a contradiction if 
$\varepsilon < \frac{\sqrt{15}-3}{6}$. 
This completes the proof of Theorem \ref{thm:finitely_many_circ}.
\end{proof}

\section{Applications} \label{sec:applications}

In this section, we prove the three applications of our method mentioned in the introduction. We only highlight where one needs to adjust the construction in each case.

\subsection{Balanced orientation} \label{subsec:balanced_orientation}

\begin{proof}[Proof of Theorem~\ref{thm:balanced_orientation}]
In this case, we have $k=1$, and we need to make sure that each $\alpha^n$ is a balanced orientation. So we start with $G_0$ a $2d$-regular finite graph and $\alpha^0$ a balanced orientation. Our task is to perform the surgery and extension steps such that $\alpha^{n+1}$ becomes a balanced orientation of $G_{n+1}$. 

This is less work than in the general case. 
The deficiency lemma tells us that on each $V_\mathbf{m}$ with at most one tight coordinate the total number of missing in-degrees and out-degrees is the same. So we could introduce additional edges within these $V_\mathbf{m}$ together with their $\alpha^{n+1}$ values such that we get a balanced orientation. However, in order to extend the homomorphism $p_0^n$, we replace every additional edge by $\texttt{Gad}(m)$, where $m$ depends only on the image of the endvertices under $p_0^n$.

Note that the used gadget $\texttt{Gad}(m)$ admits a balanced orientation everywhere, but at the vertices $u$ and $v$ with degree one: one of them will have an incoming edge, while the other will have an outgoing edge. Thus, if we replace an edge by $\texttt{Gad}(m)$ (i.e., we remove the edge $(x,y)$, add an isomorphic copy of $\texttt{Gad}(m)$ and connect $x$ to $u$ and $y$ to $v$) in a graph with a balanced orientation then it extends to a balanced orientation of the resulting graph. 

We also work collectively with the union of the remaining $V_{\mathbf{m}}$ with more than one tight coordinate denoted by $V_0^+$. This is again less work, since we do not need to attach trees to the edges on the boundary. The rest of the construction is the same. The inverse limit $\mcT$ admits a measurable balanced orientation, since the balanced orientations for the finite $\sigma$-algebras stabilize at a.e. vertex of $\mcT$, but $\mcT$ admits no other balanced orientations, moreover no other circulations in $L^2(E(\mcT),\eta)$.
\end{proof}

\subsection{Schreier graphing}
\label{subsec:Shreier_graphing}
\begin{proof}[Proof of Theorem~\ref{thm:Schreier_graphing}]
    In this case $k=d$, and we need to ensure that each $\alpha^n_i$ defines the $\mathbb{Z}$-action corresponding to the $i^{th}$ generator of $\mathbb{F}_d$. We start with a $2d$-regular graph $G_0$ and a Schreier decoration $\alpha^1_i$. Our task is to perform the surgery and extension steps such that the $\alpha^{n+1}_i$ become a Schreier decoration of $G_{n+1}$.

    Note that the used gadget $\texttt{Gad}(m)$ admits a Schreier decoration everywhere but at the vertices $u$ and $v$ with degree one: one of them will have an incoming edge, while the other will have an outgoing edge with the same label. Thus, if we replace an edge by $\texttt{Gad}(m)$ in a graph with a Schreier decoration, then it extends to a Schreier decoration of the resulting graph. The rest of the construction is the same as in the proof of Theorem \ref{thm:finitely_many_circ}.
    
    The inverse limit $\mcT$ admits a measurable Schreier decoration, since the Schreier decorations for the finite $\sigma$-algebras stabilize for a.e. vertex of $\mcT$, and $\mcT$ admits no other Schreier decorations and no other circulations in $L^2(E(\mcT),\eta)$, but the linear combinations of those corresponding to the generators.
    \end{proof}

\subsection{Proper edge coloring}
\label{subsec:edge_coloring}
\begin{proof}[Proof of Theorem~\ref{thm:perfect_matching}]

In this case $k=d$, and we construct $\mcT$ such that it is measurably bipartite, and admits a proper edge coloring $\texttt{col}:E( \mcT) \to [d]$. We write $B,W$ for the two classes of the bipartition (black, white). Each color in $[d]$ defines a circulation: for $(u,v)\in E(\mcT)$ with $u \in B$, $v \in W$ let 
\begin{equation} \label{eqn:circulation_from_coloring}
    \alpha_i(u,v)=
    \left\{
        \begin{tabular}{cc}
            $1$ & \textrm{ if } $\texttt{col}(u,v)=i$ \\
            $-1/(d-1)$ & \textrm{ if } $\texttt{col}(u,v) \neq i$  
        \end{tabular}
    \right. 
\end{equation}

Similary, any other perfect matching $M$ would define a circulation, so we aim to construct $\mcT$ such that the $\alpha_i$ above span the space of circulations on $\mcT$. We start with a connected $d$-regular bipartite graph $G_1$ with bipartition $(B_1, W_1)$ of the vertex set, and a proper edge coloring $\texttt{col}_1:E(G_1) \to [d]$. At each stage we work both with the coloring $\texttt{col}_n$ and the corresponding circulations $\alpha^n_i$ defined analogously to equation~\ref{eqn:circulation_from_coloring}. We construct a sequence of bipartite graphs (unlike in the proof of the other theorems), hence the limit will have a clopen bipartition. We will make sure that at each step we define $\texttt{col}_{n+1}$ to be a proper edge coloring of $G_{n+1}$. This will also define $\alpha^{n+1}_{i}$, and we will have $\langle \alpha_i^{n}, \alpha_j^{n} \rangle = -1/(d-1)^2$ for all $n$ and $i \neq j$.

The bipartition $(B_n, W_n)$ of $V(G_n)$ lifts trivially to $V_0(G_{n+1})$, let us denote it by $(B_{n+1}, W_{n+1})$. (When we add vertices during the surgery, we will indicate which color class they get assigned to.)  By the deficiency lemma we have $ \texttt{def}_{\alpha^{n+1}_i}(V_{\mathbf{m}}) =0$ for each $\mathbf{m}$ with one tight coordinate. Fix such an $\mathbf{m}$, and let $b_i$ denote the number of edges of color $i$ missing at vertices of $V_{\mathbf{m}} \cap B_{n+1}$, and $w_i$ denote the number of edges of color $i$ missing at vertices of $V_{\mathbf{m}} \cap W_{n+1}$. We have
\[0= \texttt{def}_{\alpha^{n+1}_i}(V_{\mathbf{m}}) = -b_i + w_i + \sum_{j \neq i} \frac{b_j}{d-1}-\frac{w_j}{d-1}.\]
Introducing $z_i=w_i-b_i$ we get 
\[0= z_i - \frac{1}{d-1}\sum_{j\neq i}z_j, \ \forall i\in \{1,\ldots,d\}.\]
That is, $z_1=z_2=\ldots=z_d = z$. This means that the excess number of missing edges from white vertices is the same for all colors. (This might be negative, meaning that the true excess is at black vertices.) Hence we can perform the surgery at $V_{\mathbf{m}}$ by first adding as many edges as possible among the vertices of $V_{\mathbf{m}}$, and then adding $z$ vertices (to $B_{n+1}$ if $z>0$, and to $W_{n+1}$ if $z <0$) to create pairs, in each edge-color-class, for those vertices of $V_{\mathbf{m}}$ that still need them. The value of $z$ depends on $\mathbf{m}$, but $|z| \leq |V_{\mathbf{m}}|$, so we are not adding too many vertices. Next, replace every additional edge by a gadget $\texttt{Gad}(m)$ for a well-chosen $m$ to guarantee that the homomorphism $p_0^n$ can be extended to the gadget. 
In the general theorem, we assumed the $G_n$ specifically \emph{not} to be bipartite. We did this to ensure that there is a path of odd length between the images of two vertices that we want to connect via a gadget. Here, on the other hand, $G_n$ is bipartite. However, this causes no problem, as we always insert gadgets between vertices that are mapped to opposite color classes, so their distance in $G_n$ is odd.

We argue similarly for the union of the remaining $V_{\mathbf{m}}$, those with at least two tight coordinates. We can introduce new vertices and edges such that the coloring is correct, but the edges might be too long. Therefore, we replace them by gadgets $\texttt{Gad}(m)$ to be able to extend the homomorphism $p_0^n$ to them, and choose $m > K_n$ in order to extend each potential $h_{\beta}$ such that its gradient remains bounded by $M_n$. The edge-coloring and the vertex-coloring can both be extended to the gadget $\texttt{Gad}(m)$. The same is true when we perform cross-surgeries to ensure connectivity.

The inverse limit $\mcT$ admits a measurable edge coloring, since the edge colorings for the finite $\sigma$-algebras stabilize for a.e. vertex of $\mcT$, and $\mcT$ admits no other measurable perfect matchings in $L^2(E(\mcT))$, but those corresponding to the generators.
\end{proof}

\section{Further directions}\label{sec:further_directions}

\paragraph{Perfect matchings in actions of $\mathbb{F}_d$.} It is natural to ask what the minimum number of circulations or perfect matching is on a Schreier graphing of a free pmp action of a finitely generated groups with acyclic Cayley graph, namely $\mathbb{F}_d$ or $(\Z / 2\Z)^{*d}$. Three out of the four possible questions are now resolved. The first author proved that $(\Z / 2\Z)^{*d}$ admits a free pmp action with no nonzero circulation \cite{kun2021measurable}. Our theorems show that $(\Z / 2\Z)^{*d}$ admits a free pmp action with exactly $d$ measurable perfect matchings, while $\mathbb{F}_d$ admits a free pmp action such that the subspace of measurable circulations in $L^2(E(\mcT),\eta)$ is exactly $d$ dimensional. That is, the possible minimum can be realized in these three cases.

This leads naturally to the following question asked by Thornton. 

\begin{question}[Thornton]
Is there an essentially free pmp action of $\mathbb{F}_d$ such that its Schreier graphing has no measurable perfect matchings?
\end{question}

Note that the Schreier graphing of such an action can not be measurably bipartite: if it was it would already have $2d$ perfect matchings corresponding to the generators and their inverses. Therefore, one cannot simply reduce the problem to circulations and apply our techniques. We expect a positive answer, but this seems to require new ideas.

\paragraph{Groups with minimum circulations.} Consider an arbitrary finitely generated group $\Gamma$. How small can the subspace of circulations in $L^2(E(\mathcal{G}),\eta)$ of the Schreier graphing $\mathcal{G}$ of a free pmp action of $\Gamma$ be?
We know that the cycle space and the circulations corresponding to the generators of infinite order are in it, and in the case of $\mathbb{F}_d$ and $(\Z / 2\Z)^{*d}$, we can construct an example where these generate all circulations. This tempts us to pose the following question.

\begin{question}
Which groups admit a finite set of generators and a free pmp action such that for the Schreier graphing $\mathcal{G}$ every circulation in $L^2(E(\mathcal{G}),\eta)$ is in the closure of the subspace generated by the cycles and the circulations corresponding to the generators?
\end{question}

The higher dimensional, homology theoretic analogues of the question regarding "complexings" also seem interesting.

\paragraph{Markov spaces.} The techniques of this paper work in a more general context. One tempting possibility, already showcased in \cite{kun2021measurable}, is to consider the inverse limit of graphs with unbounded degree in order to get so-called Markov spaces introduced in \cite{lovasz2021flows} as limit objects. (Markov spaces, generalizing graphings and graphons, admit a compatible vertex measure and edge measure, while circulations become signed measures in this context.) This looks to be a promising research direction, as one can  control the inverse limit in this case. One would hope that a good notion of rank to measure the (possibly fractional) dimension of the circulation space, in proportion to the edge or vertex space, could lead to interesting theorems.

\medskip

Recently, Bernshteyn \cite{bernshteyn2023probabilistic} and Brandt, Chang, Grebík, Grunau, Rozho\v{n}, and Vidny\'anszky \cite{brandt2021local} have discovered fascinating connections of measurable (Borel/Baire/continuous) combinatorics to local algorithms. It would be interesting to see applications of our methods in this context.

\bibliographystyle{alpha}  
\bibliography{references}

\bigskip
\noindent
{\bf Gábor Kun}\\
HUN-REN Alfréd Rényi Institute of Mathematics, Budapest, Hungary\\
and Eötvös Loránd University, Budapest, Hungary\\
\texttt{kun.gabor@renyi.hu}
\medskip
\ \\
{\bf László Márton Tóth}\\
HUN-REN Alfréd Rényi Institute of Mathematics, Budapest, Hungary\\
\texttt{toth.laszlo.marton@renyi.hu}

\end{document}